\documentclass[12pt]{article}
\usepackage{amsmath,amssymb, amsfonts,amsthm,fancyhdr,mathrsfs,geometry,graphicx,setspace,float, eufrak}
\usepackage{enumerate}

\title{\textbf{ A generalised mid summability in Banach spaces
}}
\author{Aleena Philip, Deepika Baweja}
\date{\small\textit{Department of Mathematics, BITS Pilani Hyderabad, 500078\\%
		aleenaphilip11@gmail.com, deepika@hyderabad.bits-pilani.ac.in}}

\newtheorem{thm}{Theorem}[section]
\newtheorem{pro}[thm]{Proposition}

\newtheorem{re}[thm]{Remark}
\newtheorem{Definition}[thm]{Definition}
\newtheorem{Lemma}[thm]{Lemma}

\begin{document}
	\maketitle	
	\begin{abstract}
		In this paper, we study the notion of mid summability in a general setting using the duality theory of sequence spaces. We define the vector valued sequence space $\lambda^{mid}(X)$ corresponding to a Banach space $X$ and sequence space $\lambda$. We prove that $\lambda^{mid}(\cdot)$ can be placed in a chain with the vector valued sequence spaces $\lambda^{s}(\cdot)$ and $\lambda^{w}(\cdot)$. Consequently, we define mid $\lambda$-summing operators and obtain the maximality of these operator ideals for a suitably restricted $\lambda$. Furthermore, we define a tensor norm using the vector valued sequence spaces $\lambda^{s}(\cdot)$ and  $\lambda^{mid}(\cdot)$, and establish its correspondence with the operator ideal of absolutely mid $\lambda$-summing operators.
	\end{abstract}
	\\ MSC Code 2010: 46A45, 46B28, 46B45,  47B10.
	\section{Introduction}

	The concept of mid/operator $p$-summability of sequences in Banach spaces was introduced by Karn and Sinha\cite{Karn} using the notion of absolutely $p$-summing operators. This new type of summability not only differs from absolute and weak summability, but also lies intermediate to them. As a consequence, a few classes of operators which are characterised by the transformation of sequences from/into the space of mid $p$-summable sequences, collectively called mid $p$-summing operators were introduced(see \cite{Botelho,Campos,Fourie}). It should be noted that, the classes of mid $p$-summable sequences as well as mid $p$-summing operators are  fundamentally linked to the sequence spaces  $\ell_{p}$ and their duality theory.  In this paper, we obtain a generalisation of  the sequence space  $\ell_{p}^{mid}(\cdot)$ and the corresponding ideals of mid $p$-summing operators by replacing $\ell_{p}$ for an arbitrary sequence space $\lambda$.

	In section 3, we define the vector valued sequence space $\lambda^{mid}(X)$ corresponding to a Banach space $X$ and an arbitrary sequence space $\lambda$. We establish a general result regarding the norm iteration property of the norm, $\|\cdot\|_{\lambda}$, of a sequence space $\lambda$ which has a significant impact on the results presented in this paper. We obtain few characterisations for the sequences in $\lambda^{mid}(X)$ in terms of $\lambda$-limited operators as well as $\lambda$-summing operators. Section 4 is devoted to the study of mid $\lambda$-summing operators.  We introduce the operator classes of absolutely and weakly $\lambda$-summing operators characterised by the transformation of sequences from/into the space $\lambda^{mid}(X)$. We prove that these two operator classes are  maximal Banach operator ideals for a suitably restricted $\lambda$. Finally in the last section, we obtain a tensor norm  associated to the maximal operator ideal of absolutely mid $\lambda$-summing operators. The results of this paper generalise some of the results proved in \cite{Botelho,Campos,Fourie,Karn,AP}, and hold for several sequence spaces including $\ell_{p}$ spaces, Orlicz, Modular, Lorentz spaces, cf.\cite{kamgup}, the sequence spaces $\mu_{a,p}$ and $\nu_{a,p}$ given by Garling\cite{Garling},etc.
	\section{Preliminaries}
	We shall use the letters $X,Y$ to denote Banach spaces over the field $\mathbb{K}$ of real and complex numbers. For a Banach space $(X,\|\cdot\|)$, the symbol  $X^*$ denotes its topological dual  and $B_{X}$ to denote the closed unit ball of $X$. We denote by $\mathcal{L}(X,Y)$, the space of continuous linear operators from $X$ to $Y$.

	Let $\omega$ denote the vector space of all scalar sequences defined over the field $\mathbb{K}$ with respect to the usual vector addition and scalar multiplication. The symbol $e_{n}$ represents the $n^{th}$ unit vector in $\omega$ and $\phi$ is the vector subspace spanned by the set $\{e_{n}:n\geq 1\}$. A \textbf{\textit{sequence space}} $\lambda$ is a subspace of $\omega$ such that $\phi\subseteq\lambda$. A sequence space $\lambda$ is said to be (i)\textbf{\textit{normal}} if $(\beta_{n})_{n}\in\lambda$, whenever $|\beta_{n}|\leq|\alpha_{n}|, \forall n\in\mathbb{N}$ and some $(\alpha_{n})_{n}\in\lambda$, and (ii)\textbf{\textit{symmetric}} if $(\alpha_{\pi(n)})_{n}\in\lambda$ for all permutations $\pi$ whenever $(\alpha_{n})_{n}\in\lambda$. The \textbf{\textit{cross dual}} or \textbf{\textit{K$\ddot{o}$the dual}} of $\lambda$ is the sequence space $\lambda^{\times}$ defined as
	$$\lambda^{\times}=\{(\beta_{n})_{n}\in\omega:\sum_{n}|\alpha_{n}||\beta_{n}|<\infty, \forall (\alpha_{n})_{n}\in\lambda\}.$$
	If $\lambda=\lambda^{\times\times}$, then $\lambda$ is called a \textbf{\textit{perfect sequence space}}.
	A sequence space $\lambda$ equipped with a linear topology is said to be a \textbf{\textit{K-space}} if each of  the projections $P_{n}:\lambda\rightarrow\mathbb{K}$, such that $P_{n}((\alpha_{i})_{i})=\alpha_{n}$, are continuous. A Banach $K$-space $(\lambda,\|\cdot\|_{\lambda})$ is called a \textbf{\textit{BK-space}}. For a normal $BK$-space, there  exists an equivalent norm $\||\cdot\||$ on $\lambda$ such that if $|\alpha_{n}|\leq|\beta_{n}|~\forall n$, then $\||(\alpha_{n})_{n}\||\leq\||(\beta_{n})_{n}\||$. Therefore, we shall always assume that a normal sequence space $\lambda$ is equipped with such a norm.   A BK-space  $(\lambda,\|\cdot\|_{\lambda})$ is said to be an \textbf{\textit{AK-space}} if $\sum\limits_{n=1}^{m}\alpha_{n}e_{n}$ converges to $(\alpha_{n})_{n}$ for every  $(\alpha_{n})_{n}\in\lambda$. We always assume that $0<\sup_{n}\|e_{n}\|_{\lambda}<\infty$. For a BK-space $(\lambda,\|\cdot\|_{\lambda})$, the space $\lambda^{\times}$ is a BK-space endowed with the norm, $\|(\beta_{n})_{n}\|_{\lambda^{\times}}=\sup\limits_{(\alpha_{n})_{n}\in B_{\lambda}}\{\sum_{n=1}^{\infty}|\alpha_{n}\beta_{n}|\}.$ For a normal $AK-BK$ sequence space $\lambda$, we have $(\lambda^{*},\|\cdot\|)\cong(\lambda^{\times},\|\cdot\|_{\lambda^{\times}}).$ In addition, if $\lambda$ is reflexive, then $\lambda$ is perfect.
	
	For a sequence space $(\lambda,\|\cdot\|_{\lambda})$, the norm $\|\cdot\|_{\lambda}$ is said to be
	\textbf{\textit{$k$-symmetric}} if $\|(\alpha_{n})_{n}\|_{\lambda}=\|(\alpha_{\pi(n)})_{n}\|_{\lambda}$ for all permutations $\pi$.  
	
	The vector valued sequence spaces $\lambda^{s}(X)$ and $\lambda^{w}(X)$ defined corresponding to a sequence space $\lambda$ and a Banach space $X$ were introduced by Pietsch in \cite{Pietsch}. Indeed, $\lambda^{s}(X)=\{(x_{n})_{n}\subset X: (\|x_{n}\|)_{n}\in \lambda\}$ and  $\lambda^{w}(X)=\{(x_{n})_{n}\subset X: (f(x_{n}))_{n}\in \lambda, \forall f\in X^*\}$. The spaces $\lambda^{s}(X)$ and $\lambda^{w}(X)$ are equipped with the norms $\|(x_{n})_{n}\|_{\lambda}^{s}=\|(\|x_{n}\|)_{n}\|_{\lambda}$ and $	\|(x_{n})_{n}\|_{\lambda}^{w}=\sup\limits_{f\in B_{X^{*}}} \|(f(x_{n}))_{n}\|_{\lambda}$ respectively.

	Corresponding to a sequence space $\lambda$ and the dual $X^*$ of a Banach space $X$, the vector valued sequence space $\lambda^{w^*}(X^*)$ was studied by Fourie in \cite{Fourie}.
	$\lambda^{w^{*}}(X^{*})=\{(f_{n})_{n}\subset X^{*}: (f_{n}(x))_{n}\in \lambda, \forall x\in X\}$.  	$\lambda^{w^{*}}(X^{*})$ is a  Banach space endowed with the norm
	\begin{equation*}
		\|(f_{n})_{n}\|_{\lambda}^{w^*}=\sup_{x\in B_{X}} \|(f_{n}(x))_{n}\|_{\lambda}.
	\end{equation*}
	Also for a normal $AK-BK$ sequence space $\lambda$ and Banach space $X$, we have 
	\begin{equation}\label{lamba weak star}
		\lambda^{w^{*}}(X^{*})\cong\mathcal{L}(X,\lambda).
	\end{equation}
	Further restricting $\lambda$ to be reflexive, we have
	\begin{equation}\label{lambda weak}
		\lambda^{w}(X)\cong\mathcal{L}(\lambda^{\times},X).
	\end{equation}
	In \cite{alee}, the authors proved the following theorem, which gives the conditions for which the vector valued sequence spaces $\lambda^{w}(X^{*})$ and $\lambda^{w^{*}}(X^{*})$ coincide.
	\begin{thm}\label{equality}
		Let $(\lambda,\|\cdot\|_{\lambda})$ be a normal reflexive $AK-BK$ sequence space  satisfying the condition $0<\sup_{n}\|e_{n}\|_{\lambda}<\infty$. Then for a Banach space $X$, $\lambda^{w}(X^{*})=\lambda^{w^{*}}(X^{*}).$
	\end{thm}
	In \cite{Fourie 23}, Fourie and Zeekoi introduced the \textbf{\textit{right extension}} of a Banach operator ideal $\mathcal{I}$ defined corresponding to a sequence space $\lambda$ whose components are given by  $\mathcal{I}_{\lambda}(X,Y)=\{T\in\mathcal{L}(X,Y): ST\in\mathcal{I}(X,\lambda), \forall S\in\mathcal{L}(Y,\lambda)\}$. 
	
	\begin{pro}\cite{Fourie 23}\label{zee}
		Let $\lambda$ be a normal $AK-BK$ sequence space for which $\|e_{n}\|_{\lambda}=1$ for each $n\in\mathbb{N}$. Then	$\mathcal{I}_{\lambda}$ is a Banach operator ideal with respect to the norm
		$$\mathfrak{i}_{\lambda}(T)=\sup_{S\in B_{\mathcal{L}(Y,\lambda)}}\mathfrak{i}(ST).$$
		Moreover,
		$(\mathcal{I}_{\lambda},\mathfrak{i}_{\lambda})$ is maximal, if $\mathcal{I}$ is maximal.  
	\end{pro}	
	For an arbitrary sequence space $\lambda$, Ramanujan\cite{MS} initiated the study of absolutely $\lambda$-summing operators.
	\begin{Definition}
		A continuous linear operator $T:X\rightarrow Y$ is said to be absolutely $\lambda$-summing if for each $(x_{n})_{n}\in\lambda^{w}(X)$, the sequence $(Tx_{n})\in\lambda^{s}(Y)$. 
	\end{Definition}
	The space of all absolutely $\lambda$-summing operators from $X$ to $Y$, denoted by $\Pi_{\lambda}(X,Y)$ is a Banach space endowed with norm $$\pi_{\lambda}(T)=\sup\limits_{(x_{n})_{n}\in B_{\lambda^{w}(X)}}\|(Tx_{n})_{n}\|_{\lambda}^{s}.$$ Moreover $(\Pi_{\lambda},\pi_{\lambda})$ is a Banach operator ideal.

	Further, Kim et al.\cite{Kim}  introduced the concept of absolutely $E$-summing operators, where $E$ is a  Banach space with $1$-unconditional basis. For a normal $AK-BK$ sequence space $(\lambda,\|\cdot\|)$, absolutely $\lambda$-summing operators becomes absolutely $E$-summing for $E=\lambda$. Therefore we have 
	\begin{pro}\label{max of E-summing}\cite[Theorem 3.2]{Kim}
		Let $\lambda$ be a symmetric normal   sequence space such that $(\lambda,\|\cdot\|_{\lambda})$ is an $AK-BK$ space with $\|e_{n}\|_{\lambda}=1~~\forall n$ . Then the ideal $(\Pi_{\lambda},\pi_{\lambda})$ is maximal. 
	\end{pro}
	We refer to \cite{Pietsch2} for the theory of operator ideals and \cite{kamgup} for the theory of sequence spaces.

	\section{$\lambda^{mid}(X)$}
	In this section, we study the vector valued sequence space,  $\lambda^{mid}(X)$ which is a generalization of the class of mid $p$-summable sequences for an arbitrary sequence space $\lambda$. 
	
	Corresponding to a sequence space $\lambda$ and a Banach space $X$, we define  
	\begin{Definition}
		A sequence $(x_{n})\subset X$ is said to be in $\lambda^{mid}(X)$ if for every $(f_{j})_{j}\in\lambda^{w^*}(X^*)$, $\left((f_{j}(x_{n}))_{j}\right)_{n}\in\lambda^{s}(\lambda)$.
	\end{Definition}
	Note that, we recover the space $\ell_{p}^{mid}(X)$  for $\lambda=\ell_{p}$.
	\begin{pro}
		For a Banach space $X$ and a normal sequence space $\lambda$, we have 
		$$\lambda^{s}(X)\subset \lambda^{mid}(X).$$
	\end{pro}
	\begin{proof}
		Consider $(x_{n})\in\lambda^{s}(X)$. Then for each $(f_{j})\in\lambda^{w^*}(X^*)$, $(f_{j}(x_{n}))_{j}\in\lambda~~\forall n$.
		Therefore, $\|(f_{j}(x_{n}))_{j}\|_{\lambda}=\|x_{n}\|\left\|(f_{j}(\frac{x_{n}}{\|x_{n}\|}))_{j}\right\|_{\lambda}\leq\|x_{n}\|\|(f_{j})_{j}\|_{\lambda}^{w^*}$ for each $n\in\mathbb{N}$.
		\\Since $\lambda$ is normal  and $(\|x_{n}\|)_{n}\in\lambda$, 
		$$\left(\|(f_{j}(x_{n}))_{j}\|_{\lambda}\right)_{n}\in\lambda$$
		which implies  $(x_{n})\in\lambda^{mid}(X)$.
	\end{proof}
	If we further restrict $\lambda$, we have
	\begin{pro}
		Let $X$ be a Banach space and  $(\lambda,\|\cdot\|_{\lambda})$ be a normal $AK-BK$  space. Then 
		$$\lambda^{mid}(X) \subset \lambda^{w}(X).$$
	\end{pro}
	\begin{proof}
		Let $(x_{n})_{n}\in\lambda^{mid}(X)$. Then for each $f\in X^{*}$, we have $(f,0,0,\cdots,0,\cdots)\in\lambda^{w^{*}}(X^*)$ and hence  $(f(x_{n})e_{1})_{n}=\left((f(x_{n}),0,0,\cdots,0,\cdots)\right)_{n}\in\lambda^{s}(\lambda)$.
		Since $f(x_{n})= \frac{1}{\|e_{1}\|_{\lambda}}f(x_{n}) \|e_{1}\|_{\lambda}\leq \frac{1}{\|e_{1}\|_{\lambda}}\|f(x_{n})e_{1}\|_{\lambda}$ and  $\lambda$ is normal, it follows that $(f(x_{n}))_{n}\in\lambda$ for each $f\in X^*$. Thus  $(x_{n})\in\lambda^{w}(X)$.
	\end{proof}
	It is clear that, similar to $\ell_{p}^{mid}(\cdot)$, the space $\lambda^{mid}(\cdot)$ occupies the intermediate position between  the vector valued sequence spaces $\lambda^{s}(\cdot)$ and $\lambda^{w}(\cdot)$,
	and hence justifies the name $\lambda^{mid}(\cdot)$.
	Next we equip $\lambda^{mid}(X)$ with a norm and prove that it is complete with respect to this norm.
	\begin{pro}
		Let $X$ be a Banach space and $\lambda$ be a normal $AK-BK$ sequence space. Then $\lambda^{mid}(X)$ becomes a Banach space with respect to the norm 
		$$\|(x_n)_{n}\|_{\lambda}^{mid}=\sup_{(f_{k})_{k}\in B_{\lambda^{w^*}(X^*)}}\left\|\left((f_{k}(x_{n})_{k})_{n}\right)\right\|_{\lambda}^{s}.$$
	\end{pro}
	\begin{proof}
		To prove the finiteness of 	$\|(x_n)_{n}\|_{\lambda}^{mid}$,  consider the linear map  corresponding to $\bar{x} = (x_n)_{n}\in\lambda^{mid}(X)$, $T_{\bar{x}}:\lambda^{w^*}(X^*)\rightarrow\lambda^{s}(\lambda)$ defined as  $T_{\bar{x}}((f_{j})_{j})=\left((f_{j}(x_{n}))_{j}\right)_{n}$. A straightforward application of Closed Graph theorem proves that $T_{\bar{x}}$ is bounded and hence its norm $\|T_{\bar{x}}\|=\|(x_n)_{n}\|_{\lambda}^{mid}$ is finite.
		\\For proving that $\|(x_n)_{n}\|_{\lambda}^{mid}$ is a norm, let $\|(x_n)_{n}\|_{\lambda}^{mid}=0$. Then for each $(f_{j})_{j}\in\lambda^{w^*}(X^*)$, $\left\|\left(\left\|\left(f_{j}(x_{n})\right)_{j}\right\|_{\lambda}\right)_{n}\right\|_{\lambda}=0$. Since $\|\cdot\|_{\lambda}$ is a norm, we have
		$$f_{j}(x_{n})=0, ~~\forall~~n,j\in\mathbb{N}.$$
		If we consider $(f,0,0,\cdots)\in\lambda^{w^*}(X^*)$ corresponding to each $f\in X^*$, we get $f(x_{n})=0~~\forall n$, and thereby proving that $x_{n}=0$ for each $n\in\mathbb{N}$.
		Remaining properties of the norm follows easily and canonical
		arguments prove the completeness of the space $\lambda^{mid}(\cdot)$.
	\end{proof}
	Therefore, the following chain is easily verified.
	\begin{pro}
		Let $X$ be a Banach space and  $(\lambda,\|\cdot\|_{\lambda})$ be a normal $AK-BK$ space. Then 
		$$\lambda^{s}(X)\overset{1}{\hookrightarrow} \lambda^{mid}(X)\overset{1}{\hookrightarrow} \lambda^{w}(X).$$
	\end{pro}
	
	The following result immediately follows from the relation $\lambda^{w^*}(X^*)\cong\mathcal{L}(X,\lambda)$.
	\begin{pro}\label{operator definition}
		Let $X$ be a Banach space, and $(\lambda,\|\cdot\|_{\lambda})$ be a normal $AK-BK$  sequence space. Then $(x_{j})_{j}\in\lambda^{mid}(X)$ if and only if $(Tx_{j})_{j}\in\lambda^{s}(\lambda)$ for every $T\in\mathcal{L}(X,\lambda)$.
	\end{pro}
	Moreover we have,
	$$\|(x_n)_{n}\|_{\lambda}^{mid}=\sup_{(f_{k})_{k}\in B_{\lambda^{w^*}(X^*)}}\left\|\left((f_{k}(x_{n})_{k})_{n}\right)\right\|_{\lambda}^{s}=\sup_{T\in B_{\mathcal{L}(X,\lambda)}}\|(Tx_{j})_{j}\|_{\lambda}^{s}.$$
	\begin{re}
		In \cite{Fourie 23}, Fourie and Zeekoi defined the space of operator $[Y,E]$-summable sequences in a Banach space $X$ as
		$$Y_{E}(X)=\{(x_{n})_{n}\subset X:\sum\limits_{n=1}^{\infty}\|Tx_{n}\|e_{n}<\text{ converges in } E,~\forall~T\in\mathcal{L}(X,Y)\}$$
		where $Y$ is a non trivial Banach space and $E$ is a Banach space with $1$-unconditional basis $\{e_{n}:n\in\mathbb{N}\}$.
		Note that, when $\|e_{n}\|_{\lambda}=1~~\forall n$,  the sequence space $\lambda^{mid}(X)$ coincides with the space of operator $[Y,E]$ summable sequences for $Y=E=\lambda$.
	\end{re}
Recall from \cite{MSnip}
\begin{Definition}
	The norm, $\|\cdot\|_{\lambda}$, of a sequence space $\lambda$ is said to have the norm iteration property if  for each sequence $(\bar{\alpha}_{n})_{n}=((\alpha_{j}^{n})_{j})_{n}$ in $\lambda^{s}(\lambda)$, then the sequences $\bar{\alpha}_{j}=(\alpha_{j}^{n})_{n}\in\lambda, \forall j\in\mathbb{N}$, $(\bar{\alpha}_{j})_{j}\in\lambda^{s}(\lambda)$
	  and $\|(\bar{\alpha}_{n})_{n}\|_{\lambda}^{s}=\|(\bar{\alpha}_{j})_{j}\|_{\lambda}^{s}.$
\end{Definition}
One can note the significant role played by the norm iteration property of sequence spaces on the theory of $\lambda^{mid}(\cdot)$, particularly due to its relationship with the space $\lambda^{s}(\lambda)$. Hence, we seek to establish a broad result concerning this property of sequence spaces.
	\begin{thm}\label{nip}
		Let $(\lambda,\|\cdot\|_{\lambda})$ be a perfect $BK$-space. Then $\|\cdot\|_{\lambda}$ has the norm iteration property.
	\end{thm}
	\begin{proof}
		Since the identity map from $(\lambda,\|\cdot\|_{\lambda})$ to  $(\lambda,\|\cdot\|_{\lambda^{\times\times}})$ is continuous, $\|\cdot\|_{\lambda}$ and $ \|\cdot\|_{\lambda^{\times\times}}$ are equivalent by \cite[Corollary 1.6.8]{Megg}. Therefore, for $(\alpha_{n})_{n}\in\lambda$, we have 
		$$\|(\alpha_{n})_{n}\|_{\lambda}=\sup_{(\beta_{n})_{n}\in B_{\lambda^{\times}}}\sum_{n=1}^{\infty}|\alpha_{n}\beta_{n}|.$$
		To prove that $\|\cdot\|_{\lambda}$ has the norm iteration property, consider $(\bar{\alpha}_{n})_{n}=((\alpha_{j}^{n})_{j})_{n}$ in $\lambda^{s}(\lambda)$. Clearly,
		$\sum\limits_{n=1}^{\infty}\sum\limits_{j=1}^{\infty}|\alpha_{j}^{n}\beta_{j}\gamma_{n}|<\infty$ for $(\beta_{j})_{j},(\gamma_{n})_{n}\in\lambda^{\times}$ and $\sum\limits_{n=1}^{\infty}\sum\limits_{j=1}^{\infty}|\alpha_{j}^{n}\beta_{j}\gamma_{n}|=\sum\limits_{j=1}^{\infty}\sum\limits_{n=1}^{\infty}|\alpha_{j}^{n}\beta_{j}\gamma_{n}|$. Therefore for $(\beta_{j})_{j}=e_{1}$, $$\sum\limits_{n=1}^{\infty}\sum\limits_{j=1}^{\infty}|\alpha_{j}^{n}\beta_{j}\gamma_{n}|=\sum\limits_{n=1}^{\infty}|\alpha_{j}^{n}\gamma_{n}|<\infty.$$ Thus, $\bar{\alpha}_{j}=(\alpha_{j}^{n})_{n}\in\lambda$ for each $j\in\mathbb{N}$.

		To show that $(\bar{\alpha_{j}})_{j}\in\lambda^{s}(\lambda)$, consider $(\beta_{j})_{j}\in\lambda^{\times}$. Then for $m\in\mathbb{N}$,
		\begin{align*}
			\sum_{j=1}^{m}\|\bar{\alpha_{j}}\|_{\lambda}|\beta_{j}|&=\sum_{j=1}^{m}\left(\sup_{(\gamma_{n})_{n}\in B_{\lambda^\times}}\sum_{n=1}^{\infty}|\alpha_{j}^{n}\gamma_{n}|\right)|\beta_{j}|=\lim_{k\rightarrow\infty}\sup_{(\gamma_{n})_{n=1}^{k}\in B_{\lambda^\times}}\sum_{j=1}^{m}\sum_{n=1}^{k}|\alpha_{j}^{n}\beta_{j}\gamma_{n}|\\
			&= \lim_{k\rightarrow\infty}\sup_{(\gamma_{n})_{n=1}^{k}\in B_{\lambda^\times}}\sum_{n=1}^{k}\sum_{j=1}^{m}|\alpha_{j}^{n}\beta_{j}\gamma_{n}|=\sup_{(\gamma_{n})_{n}\in B_{\lambda^\times}}\sum_{n=1}^{\infty}\left(\sum_{j=1}^{m}|\alpha_{j}^{n}\beta_{j}|\right)|\gamma_{n}|\\
			&\leq \sup_{(\gamma_{n})_{n}\in B_{\lambda^\times}}\sum_{n=1}^{\infty}\left(\|\bar{\alpha_{n}}\|_{\lambda}\|(\beta_{j})_{j}\|_{\lambda^{\times}}\right)|\gamma_{n}|
			\leq \|(\beta_{j})_{j}\|_{\lambda^{\times}}\|(\bar{\alpha_{n}})_{n}\|_{\lambda}^{s}.
		\end{align*}
		Hence $(\bar{\alpha_{j}})_{j}\in\lambda^{s}(\lambda)$.
		Also, we have 
		$$\|(\bar{\alpha_{n}})_{n}\|_{\lambda}=\sup_{(\gamma_{n})_{n}\in B_{\lambda^\times}}\sum_{n=1}^{\infty}\left(\sup_{(\beta_{j})_{j}\in B_{\lambda^\times}}\sum_{j=1}^{\infty}|\alpha_{j}^{n}\beta_{j}|\right)|\gamma_{n}|=\sup_{(\gamma_{n})_{n}\in B_{\lambda^\times}}\sup_{(\beta_{j})_{j}\in B_{\lambda^\times}}\sum_{n=1}^{\infty}\sum_{j=1}^{\infty}|\alpha_{j}^{n}\beta_{j}\gamma_{n}|.$$
		Similarly,
		$$\|(\bar{\alpha_{j}})_{j}\|_{\lambda}=\sup_{(\beta_{j})_{j}\in B_{\lambda^\times}}\sup_{(\gamma_{n})_{n}\in B_{\lambda^\times}}\sum_{j=1}^{\infty}\sum_{n=1}^{\infty}|\alpha_{j}^{n}\beta_{j}\gamma_{n}|.$$
		It is clear that the sequence $((\alpha_{j}^{n}\beta_{j}\gamma_{n})_{j})_{n}\in\ell_{1}^{s}(\ell_{1})$ and hence $\sum\limits_{n=1}^{\infty}\sum\limits_{j=1}^{\infty}|\alpha_{j}^{n}\beta_{j}\gamma_{n}|=\sum\limits_{j=1}^{\infty}\sum\limits_{n=1}^{\infty}|\alpha_{j}^{n}\beta_{j}\gamma_{n}|$ using the norm iteration property of the norm of $\ell_{1}$.
		\\Applying \cite[Lemma 4.2]{alee} to the bounded real valued  function $g:\lambda^{\times}\times\lambda^{\times}\rightarrow\mathbb{K}$ defined as $g((\beta_{j})_{j},(\gamma_{n})_{n})=\sum\limits_{j=1}^{\infty}\sum\limits_{n=1}^{\infty}|\alpha_{j}^{n}\beta_{j}\gamma_{n}|$ gives the norm iteration property of $\|\cdot\|_{\lambda}$.
	\end{proof}
Hence we have the following sequence spaces with its norm having the norm iteration property other than $\ell_{p}$:
\begin{itemize}
	\item \textbf{Orlicz sequence space $\ell_{M}$}\cite[p.297]{kamgup}, where $M$ is an Orlicz function satisfying $\Delta_{2}$ condition defined as 
	$\{\bar{\alpha}\in\omega:\sum_{j=1}^{\infty}M\left(\frac{|\alpha_{j}|}{k}\right)<\infty\text{ for some }k>0\}$ with respect to the norm $$\|(\alpha_{j})_{j}\|_{(M)}=\inf\left\{k>0:\sum_{j=1}^{\infty}M\left(\frac{|\alpha_{j}|}{k}\right)\leq 1\right\}.$$
	\item  \textbf{Modular sequence spaces $\ell_{M_j}$}\cite[p.319]{kamgup}, where ${M_{j}}$ is a sequence of Orlicz functions defined as 
	$\{\bar{\alpha}\in\omega:\sum_{j=1}^{\infty}M_{j}\left(\frac{|\alpha_{j}|}{k}\right)<\infty\text{ for some }k>0\}$ with respect to the norm, $$\|(\alpha_{j})_{j}\|_{(M_{j})}=\inf\left\{k>0:\sum_{j=1}^{\infty}M_{j}\left(\frac{|\alpha_{j}|}{k}\right)\leq 1\right\}.$$

	\item For $1\leq p<\infty$, \textbf{Lorentz sequence spaces order $p$}\cite[p.323]{kamgup} defined as $d(x,p)=\{\bar{\alpha}\in c_{0}:\sup\{\sum_{j=1}^{\infty}x_{j}|\alpha_{\sigma(j)}|^{p}:\sigma\in\Pi\}<\infty\}$, where $x=(x_{j})_{j}\in c_{0}, x\notin\ell_{1}$ such that $x_{j}>0~~ \forall j$ and $1=x_{1}\geq x_{2}\geq\cdots$ endowed with the norm
	$$\|\bar{\alpha};p\|=\sup\{\sum_{j=1}^{\infty}x_{j}|\alpha_{\sigma(j)}|^{p}:\sigma\in\Pi\}.$$
		\item \textbf{The spaces $\mu_{a,p}$ and $\nu_{a,p}$} introduced by Garling in \cite{Garling}. Let $\bar{\alpha}=(\alpha_{j})_{j}\in c_{0}$. Then $\widehat{\bar{\alpha}}=(\widehat{\alpha_{1}},\widehat{\alpha_{2}},\cdots\widehat{\alpha_{j}},\cdots)$ where $\widehat{\alpha_{j}}=\inf\limits_{A\subset\mathbb{Z}^{+}, |A|<j}\left(\sup\limits_{j\notin A}|\alpha_{j}|\right)$. If $X$ is a symmetric sequence space, $X^{++}=\{\bar{\alpha}\in X:\bar{\alpha}=\widehat{\bar{\alpha}}\}$. For $1\leq p\leq\infty$, $q$ denotes the conjugate of $p$ and $M_{q}=(\ell_{q})^{++}\cap B_{\ell_{q}}$. Suppose $(a_{j})_{j}\in c_{0}^{++}$ such that $(a_{j})_{j}\notin\ell_{1}$ and $b_{j}=a_{j}^{1/p}$. Then 
	$$\mu_{a,p}=\left\{\bar{\alpha}\in c_{0}:\sum_{j=1}^{\infty}\widehat{\alpha_{j}}^{p}a_{j}<\infty\right\}$$ with respect to the norm,
	$$\|\bar{\alpha}\|_{a,p}=\left(\sum_{j=1}^{\infty}\widehat{\alpha_{j}}^{p}a_{j}\right)^{1/p}$$
	and $$\nu_{a,p}=\left\{\bar{\beta}=(\beta_{j})_{j}:\sup_{n}\frac{\sum_{j=1}^{n}\widehat{\beta_{j}}}{\sum_{j=1}^{n}k_{j}b_{j}}<\infty\text{ for some }(k_{j})_{j}\in M_{q}\right\}$$ with respect to the norm
	$$\||\bar{\beta}|\|_{a,p}=\inf_{(k_{j})_{j}\in M_{q}}\sup_{n}\frac{\sum\limits_{j=1}^{n}\widehat{\beta_{j}}}{\sum\limits_{j=1}^{n}k_{j}b_{j}}$$
	are reflexive symmetric $BK$-spaces which are Kothe duals of each other. 
	\item \textbf{The sequence spaces $m(\bar{\phi})$} and $n(\bar{\phi})$ introduced by Sargent in \cite{Sargent}. For $\bar{\alpha}=(\alpha_{j})_{j}$, define the sequence $\Delta_{\bar{\alpha}} = (\alpha_{j}-\alpha_{j-1})$, $\alpha_{0} = 0$; $S(\bar{\alpha})$ denotes the
	collection of all sequences which are permutations of $\bar{\alpha}$. $\mathfrak{C}$ is the set of all finite
	sequences of positive integers. For $\sigma\in\mathfrak{C}$ define $c(\sigma)=(c_{j}(\sigma))$, where $c_{j}(\sigma)= 1$
	if $j\in\sigma$ and $0$, otherwise. Let $\mathfrak{C}_{s}=\{\sigma\in\mathfrak{C}:\sum_{j}^{\infty} c_{j}(\sigma)\leq s\}$.
	Let $\bar{\phi}=(\phi_{j})$ is a given (fixed) sequence such that for each $j$, $0<\phi_{1}\leq\phi_{j}\leq\phi_{j+1}$
	and $(j+1)\phi_{j}>j\phi_{j+1}$. Then sequence spaces
	$$m(\bar{\phi})=\left\{\bar{\alpha}:\|\bar{\alpha}\|=\sup_{s\geq 1}\sup_{\sigma\in\mathfrak{C}_{s}}\left(\frac{1}{\phi_{s}}\sum_{j\in\sigma}|\alpha_{j}|\right)<\infty\right\}$$ and
	$$n(\bar{\phi})=\left\{\bar{\alpha}:\|\bar{\alpha}\|=\sup_{u\in S(\bar{\alpha})}\sum_{j}|u_{j}|\Delta_{\phi_{j}}<\infty\right\}$$	are perfect $BK$-spaces which are Kothe duals of each other. 
\end{itemize}
\begin{re}
	Note that the converse of Theorem \ref{nip} need not be true. For instance,  the sequence space $c_{0}$, comprising all sequences converging to zero with the supremum norm has the norm iteration property without being a perfect sequence space. 
\end{re}

	In \cite{alee}, the authors introduced and studied the concept of $\lambda$-limited sets and operators in Banach spaces.
	\begin{Definition}
		Let $\lambda$ be a sequence space. 
		\begin{itemize}
			\item Then a subset $A$ of a Banach space $X$ is said to be $\lambda$-limited, if for each $(f_{n})_{n}\in \lambda^{w^{*}}(X^{*})$, there exists a sequence $(\alpha_{n})_{n}\in\lambda$ such that
			$$\left|f_{n}(x)\right|\leq\alpha_{n}, ~~~~~\forall~ n\in\mathbb{N}, x\in A.$$
			\item  An operator $T:X\rightarrow Y$ is said to be $\lambda$-limited if $T(B_{X})$ is $\lambda$-limited in $Y$.
		\end{itemize}
	\end{Definition}
	The next result establishes the relation between $\lambda$-limited operators and $\lambda^{mid}(\cdot)$.
	\begin{pro}
		Let $X$ be a Banach space, and $(\lambda,\|\cdot\|_{\lambda})$ be a perfect $AK-BK$ space. Then 
		$\bar{x}=(x_{n})_{n}\in\lambda^{mid}(X)$ if and only if
		the operator $E_{\bar{x}}:\lambda^{\times}\rightarrow X$ defined as $E_{\bar{x}}((\alpha_{n})_{n})=\sum\limits_{n=1}^{\infty}\alpha_{n}x_{n}$ is $\lambda$-limited.
		
	\end{pro}
	\begin{proof}
		Let $\bar{x}=(x_{n})_{n}\in\lambda^{mid}(X)$. Since $(\lambda,\|\cdot\|_{\lambda})$ is perfect, $\|\cdot\|_{\lambda}$ has the norm iteration property. Therefore we have, $\left(\|(f_{j}(x_{n}))_{n}\|_{\lambda}\right)_{j}\in\lambda$ and $$\left\|\left(\|(f_{j}(x_{n}))_{n}\|_{\lambda}\right)_{j}\right\|=\left\|\left(\|(f_{j}(x_{n}))_{j}\|_{\lambda}\right)_{n}\right\|$$ for each $(f_{j})_{j}\in\lambda^{w^*}(X^*)$. For $\bar{\alpha}\in B_{\lambda^{\times}}$, 
		\begin{align*}
			|f_{j}(E_{\bar{x}}(\bar{\alpha}))|=\left|f_{j}\left(\sum_{n=1}\alpha_{n}x_{n})\right)\right|&= \left|\sum_{n=1}\alpha_{n}f_{j}(x_{n})\right|\\
			&\leq\|(f_{j}(x_{n}))_{n}\|_{\lambda}\|\|(\alpha_{n})_{n}\|_{\lambda^{\times}}\\
			&\leq\|(f_{j}(x_{n}))_{n}\|_{\lambda}
		\end{align*}
		implies that $E_{\bar{x}}(B_{\lambda^{\times}})$ is a $\lambda$-limited subset of $X$ and hence $E_{\bar{x}}$ is a $\lambda$-limited operator.
		\\Conversly, let $E_{\bar{x}}$ is a $\lambda$-limited operator. Then for each $(f_{j})_{j}\in\lambda^{w^*}(X^*)$, there exists $(\beta_{j})_{j}\in\lambda$ such that
		\begin{align}\label{EX-limited}
			\sup_{(\alpha_{n})_{n}\in B_{\lambda^{\times}}}\left|f_{j}\left(\sum_{n=1}\alpha_{n}x_{n})\right)\right|&\leq\beta_{j}~~~\forall j\in\mathbb{N}\nonumber\\
			\text{which implies} \sup_{(\alpha_{n})_{n}\in B_{\lambda^{\times}}}\left|\sum_{n=1}\alpha_{n}f_{j}(x_{n})\right|&\leq\beta_{j}.
		\end{align}
		Since $\lambda$ is perfect and hence normal, using  \eqref{EX-limited}, $\left(\|(f_{j}(x_{n}))_{n}\|_{\lambda}\right)_{j}\in\lambda$. Using the norm-iteration property of $\|\cdot\|_{\lambda}$, it follows that $(x_{n})_{n}\in\lambda^{mid}(X)$.
	\end{proof}
	In addition, if $\lambda$ is  reflexive, we obtain the following  characterisation for sequences in $\lambda^{mid}(X)$ in terms of absolutely $\lambda$-summing operators.
	\begin{pro}
		Let $X$ be a Banach space, and $(\lambda,\|\cdot\|_{\lambda})$ be a reflexive $AK-BK$ space. Then $(x_n)_{n}\in\lambda^{mid}(X)$ if and only if $\Phi_{\bar{x}}:X^*\rightarrow\lambda$ defined as $\Phi_{\bar{x}}(f)=(f(x_{n}))_{n}$ for $f\in X^*$ is absolutely $\lambda$-summing.
	\end{pro}
	\begin{proof}
		Let $(x_n)_{n}\in\lambda^{mid}(X)$. Then by defnition of $\lambda^{mid}(X)$ and using the norm iteration property of $\|\cdot\|_{\lambda}$, for each $(f_{j})_{j}\in\lambda^{w}(X^*)$, we have 
		\begin{align*}
			\|\Phi_{\bar{x}}((f_{j})_{j}\|_{\lambda}^{s}&=\left\|\left(\left\|\left(f_{j}(x_{n})\right)_{n}\right\|_{\lambda}\right)_{j}\right\|_{\lambda}\\
			&=\left\|\left(\left\|\left(f_{j}(x_{n})\right)_{j}\right\|_{\lambda}\right)_{n}\right\|_{\lambda}\\
			&<\infty.
		\end{align*} 
		Tracing back the proof, the converse follows easily from Proposition \ref{equality}.
	\end{proof}
	\section{Mid $\lambda$-summing operators}
	Extending the notion of mid $p$-summing operators for arbitrary sequence spaces $\lambda$, we introduce the  classes of  absolutely and weakly mid $\lambda$- summing  operators determined by the inclusion
	$$\lambda^{s}(\cdot)\subset \lambda^{mid}(\cdot) \subset \lambda^{w}(\cdot).$$
	\begin{Definition}
		A linear operator $T:X\rightarrow Y$ is said to be absolutely mid $\lambda$-summing if $(Tx_{n})_{n}\in\lambda^{s}(Y)$ whenever $(x_{n})_{n}\in\lambda^{mid}(X)$.
	\end{Definition}
	We denote by the symbol $\Pi_{\lambda}^{mid}(X,Y)$, the space of absolutely  mid $\lambda$-summing operators from $X$ to $Y$.

	\begin{pro}\label{Pi lambda mid Banach}
		Let $X$ and $Y$ be Banach spaces,  and $(\lambda,\|\cdot\|_{\lambda})$ be a normal $AK-BK$  sequence space. Then $\Pi_{\lambda}^{mid}(X,Y)$ is a Banach space endowed with the norm
		$$\pi_{\lambda}^{mid}(T)=\sup_{(x_{n})_{n}\in B_{\lambda^{mid}(X)}}\|(Tx_{j})_{j}\|_{\lambda}^{s}.$$
	\end{pro}
	We omit the proof of the above proposition since it can be established using standard arguments.
	\begin{thm}\label{finite absolute lambda mid}
		Let $X,Y$ be Banach spaces, and  $(\lambda,\|\cdot\|_{\lambda})$ be a normal symmetric $AK-BK$  sequence space. The linear map $T:X\rightarrow Y$ is absolutely mid $\lambda$-summing if and only if there exists $C>0$ such that for each finite set of elements $x_{1},x_{2},\cdots,x_{n}$ in $X$, the following inequality holds:
		\begin{equation}\label{lambda mid finite sum}
			\|(Tx_{i})_{i=1}^{n}\|_{\lambda}^{s}\leq C \|(x_{i})_{i=1}^{n}\|_{\lambda}^{mid}.
		\end{equation}
	\end{thm}
	\begin{proof}
		Let $T\in\Pi_{\lambda}^{mid}(X,Y)$ and suppose that for every $C>0$, the above inequality is not true. Therefore for each $C>0$, there exists a finite sequence of vectors $x_{1}, x_{2},\cdots,x_{n(c)}$  in $X$, such that $\|(x_{i})_{i=1}^{n(c)}\|_{\lambda}^{mid}\leq 1$ and $\|(Tx_{i})_{i=1}^{n}\|_{\lambda}^{s}>C$.
		For $C=j2^{j}, j=1,2,3,\cdots$, we can obtain finite sequences $x_{1}^{j}, x_{2}^{j},\cdots,x_{n(j)}^{j}$ such that the concatenated sequence
		$\bar{x}=\frac{x_{1}^{1}}{2}, \frac{x_{2}^{1}}{2},\cdots,\frac{x_{n(1)}^{1}}{2}, \frac{x_{1}^{2}}{2^{2}}, \frac{x_{2}^{2}}{2^{2}},\cdots,\frac{x_{n(2)}^{2}}{2^{2}},\cdots, \frac{x_{1}^{j}}{2^{j}}, \frac{x_{2}^{j}}{2^{j}},\cdots,\frac{x_{n(j)}^{2}}{2^{j}},\cdots$ belongs to $\lambda^{mid}(X)$ with $\|\bar{x}\|_{\lambda}^{mid}\leq 1$. But using the monotonicity of $\|\cdot\|_{\lambda}$, we  obtain $T\bar{x}\in\lambda^{s}(Y)$ which is a contradiction.

		Conversly assume that $T:X\rightarrow Y$ be a continuous linear operator for which  \eqref{lambda mid finite sum} holds. Consider $(x_{n})_{n}\in\lambda^{mid}(X)$. Then for any fixed $j$, we have
		\begin{align*}
			\|(Tx_{1},Tx_{2},\cdots,Tx_{j},0,0,\cdots)\|_{\lambda}^{s}&\leq C \|(x_{1},x_{2},\cdots,x_{j},0,0,\cdots)\|_{\lambda}^{mid}\\
			&\leq C \|(x_{n})_{n}\|_{\lambda}^{mid}
		\end{align*}
		since $\|\cdot\|_{\lambda}$ is monotone. Using the $AK$-property of $\lambda$ it follows that $$\|(Tx_{n})_{n}\|_{\lambda}^{s}\leq C \|(x_{n})_{n}\|_{\lambda}^{mid}.$$
	\end{proof}
	Next we prove an important property of this operator ideal which is given in
	\begin{thm}
		Let $(\lambda,\|\cdot\|_{\lambda})$ be a normal symmetric $AK-BK$  sequence space. Then $(\Pi_{\lambda}^{mid},\pi_{\lambda}^{mid})$ is a maximal Banach operator ideal.
	\end{thm}
	\begin{proof}
		To prove that $\Pi_{\lambda}^{mid}$ is an operator ideal, we first show that $\Pi_{\lambda}^{mid}(X,Y)$ contains the class of all finite rank operators for every Banach spaces $X,Y$. Consider $T\in\mathcal{L}(X,Y)$ such that $rank(T)=1$. Then $T$ can be represented as $T=f\otimes y$ for some $f\in X^*, y\in Y$. For $(x_{n})_{n}\in\lambda^{mid}(X)$, we get $(Tx_{n})_{n}=(f\otimes y(x_{n}))_{n}=(f(x_{n})y)_{n}$. Since $\lambda^{mid}(X)\subset \lambda^{w}(X)$, the sequence $(\|Tx_{n}\|)_{n}=(\|y\||f(x_{n})|)_{n}\in\lambda$ which proves  that $T\in\Pi_{\lambda}^{mid}(X,Y)$.
		\\ To establish the ideal property, consider $T\in\Pi_{\lambda}^{mid}(X,Y), S\in\mathcal{L}(X_{0},X)$  and $R\in\mathcal{L}(Y,Y_{0}).$ Then for $(x_{n})_{n}\in\lambda^{mid}(X_{0})$, it is easy to see that $(RTSx_{n})_{n}\in\lambda^{mid}(X)$ and hence $RTS\in\Pi_{\lambda}^{mid}(X_{0},Y)$. The relation $\pi_{\lambda}^{mid}(RTS)\leq\|R\|\pi_{\lambda}^{mid}(T)\|S\|$ can be obtained using straightforward calculations. The above arguments together with Proposition \ref{Pi lambda mid Banach} prove that $(\Pi_{\lambda}^{mid},\pi_{\lambda}^{mid})$  is a Banach operator ideal.

		In order to prove that $(\Pi_{\lambda}^{mid},\pi_{\lambda}^{mid})$ is maximal, we consider $T\in(\Pi_{\lambda}^{mid})^{max}(X,Y)$. For a finite subset $\{x_{1},x_{2},\cdots,x_{n}\}$ of $X$, choose  $\{f_{1},f_{2},\cdots,f_{n}\}\subset B_{Y^*}$ such that $\|Tx_{i}\|=f_{i}(Tx_{i})~~\forall i=1,2,\cdots,n.$
		Also let $M=span\{x_{1},x_{2},\cdots,x_{n}\}$  and  $L=\{y\in Y: f_{i}(y)=0, \forall i=1,2,\cdots,n\}$, and $I_{M}:M\rightarrow X$ and $Q_{L}: Y\rightarrow Y/L$ denote the inclusion map and quotient map respectively. Let $\bar{f_{i}}\in B_{(Y/L)^*}$ be such  that $Q_{L}^{*}(\bar{f_{i}})=f_{i}~~\forall i=1,2,\cdots,n.$
		\begin{align*}
			\|(Tx_{i})_{i=1}^{n}\|_{\lambda}^{s}&=\left\|\sum_{i=1}^{n}f_{i}(Tx_{i})e_{i}\right\|_{\lambda}\\
			&=\left\|\sum_{i=1}^{n}Q_{L}^{*}(\bar{f_{i}})(TI_{M}x_{i})e_{i}\right\|_{\lambda}\\
			&=\left\|\sum_{i=1}^{n}\bar{f_{i}}(Q_{L}TI_{M}x_{i})e_{i}\right\|_{\lambda}\\
			&\leq \|(Q_{L}TI_{M}x_{i})_{i=1}^{n}\|_{\lambda}^{s}\leq \pi_{\lambda}^{mid}(Q_{L}TI_{M})\|(x_{i})_{i=1}^{n}\|_{\lambda}^{mid}.
		\end{align*}
		By Proposition \ref{finite absolute lambda mid} we have,  $T\in\Pi_{\lambda}^{mid}(X,Y)$ with $\pi_{\lambda}^{mid}(T)\leq(\pi_{\lambda}^{mid})^{max}(T)$.
	\end{proof}
	Now we define the class of weakly mid $\lambda$-summing operators. 
	\begin{Definition}
		A linear operator $T:X\rightarrow Y$ is said to be weakly mid $\lambda$-summing if $(Tx_{n})_{n}\in\lambda^{mid}(Y)$ whenever $(x_{n})_{n}\in\lambda^{w}(X)$.
	\end{Definition}
	We denote by the symbol $W_{\lambda}^{mid}(X,Y)$, the space of weakly  mid $\lambda$-summing operators from $X$ to $Y$.
	The next result shows that $W_{\lambda}^{mid}$ can be viewed as the right extension of the ideal of absolutely $\lambda$-summing operators. 
	\begin{pro}\label{Weak mid}
		Let $X$ and $Y$ be Banach spaces, and $(\lambda,\|\cdot\|_{\lambda})$ be a normal $AK-BK$  sequence space. Then the following are equivalent:
		\begin{enumerate}
			\item $T\in W_{\lambda}^{mid}(X,Y)$.
			\item  $ST\in\Pi_{\lambda}(X,\lambda)$ for all $S\in\mathcal{L}(Y,\lambda)$.
		\end{enumerate}	
	\end{pro}
	\begin{proof}
		Let $T\in W_{\lambda}^{mid}(X,Y)$. Then for $(x_{n})_{n}\in\lambda^{w}(X)$, $(Tx_{n})_{n}\in\lambda^{s}(Y)$. Hence by Proposition \ref{operator definition}, $(STx_{n})_{n}\in\lambda^{s}(\lambda)$ for each $S\in\mathcal{L}(Y,\lambda)$ which proves that $ST$ is $\lambda$-summing. Converse can be proved easily by tracing back the proof.
	\end{proof} 
	\begin{re}
		Along with the conditions on $\lambda$ in Proposition \ref{Weak mid}, if we assume $\lambda$ to be a reflexive sequence, we can add the following characterization to the previous result.
	\end{re}
	\begin{enumerate}[\textit{3.}]
		\item $TR\in\Pi_{\lambda}^{d}(\lambda^{\times},Y)$, for  all $R\in\mathcal{L}(\lambda^{\times},X)$.
	\end{enumerate}
	\begin{proof}
		From the identification  $\lambda^{w}(X)\cong\mathcal{L}(\lambda^{\times},X)$, for each $R\in\mathcal{L}(\lambda^{\times},X)$, there exists a sequence $(x_{n})_{n}\in\lambda^{w}(X)$ such that $R((\alpha_{n})_{n})=\sum_{n}\alpha_{n}x_{n}$. Therefore, for $T\in W_{\lambda}^{mid}(X,Y)$, and $(f_{j})_{j}\in\lambda^{w}(Y^*)$,
		\begin{align*}
			\left\|((TR)^{*}(f_{j}))_{j}\right\|_{\lambda}^{s}&=\left\|\left(\left\|\left(f_{j}(Tx_{n})\right)_{n}\right\|_{\lambda}\right)_{j}\right\|_{\lambda}\\
			&=\left\|\left(\left\|\left(f_{j}(Tx_{n})\right)_{j}\right\|_{\lambda}\right)_{n}\right\|_{\lambda}\\
			&<\infty
		\end{align*}
		follows from the norm iteration property of $\|\cdot\|_{\lambda}$. 
	\end{proof}
	Combining Propositions \ref{zee}, \ref{max of E-summing}  and \ref{Weak mid}, we obtain the following result.
	\begin{pro}
		Let $X$ and $Y$ be Banach spaces, and $(\lambda,\|\cdot\|_{\lambda})$ be normal $AK-BK$  sequence space such that $\|e_{n}\|_{\lambda}=1~~\forall n$. Then $ W_{\lambda}^{mid}$ is a Banach operator ideal with respect to the norm
		$$w_{\lambda}^{mid}(T)=\sup_{S\in B_{\mathcal{L}(Y,\lambda)}}\Pi_{\lambda}(ST).$$
		Moreover the operator ideal $(W_{\lambda}^{mid},w_{\lambda}^{mid})$ is maximal.
	\end{pro}
	\begin{thm}
		Let $X$ and $Y$ be Banach spaces, and $(\lambda,\|\cdot\|_{\lambda})$ be normal $AK-BK$  sequence space satisfying the condition $\|e_{n}\|_{\lambda}=1~~\forall n$. Then, the linear map $T:X\rightarrow Y$ is weakly mid $\lambda$-summing if and only if for each finite set of elements $x_{1},x_{2},\cdots,x_{n}$ in $X$, the following inequality holds:
		$$\|(Tx_{i})_{i=1}^{n}\|_{\lambda}^{mid}\leq w_{\lambda}^{mid}(T)\|(x_{i})_{i=1}^{n}\|_{\lambda}^{w}$$
	\end{thm}
	\begin{proof}
		$T\in W_{\lambda}^{mid}(X,Y).$
		\\$ \iff ST\in\Pi_{\lambda}(X,\lambda)$ for every $S\in\mathcal{L}(Y,\lambda).$
		\\$\iff$ for each finite set of elements $x_{1},x_{2},\cdots,x_{n}$ in $X$, the following inequality holds:
		\begin{equation*}
			\|(STx_{i})_{i=1}^{n}\|_{\lambda}^{s}\leq \pi_{\lambda}(ST) \|(x_{i})_{i=1}^{n}\|_{\lambda}^{w}.
		\end{equation*}  for every $S\in\mathcal{L}(Y,\lambda)$.
		Taking supremum over $B_{\mathcal{L}(Y,\lambda)}$ on both sides of the above inequality gives
		$$\|(Tx_{i})_{i=1}^{n}\|_{\lambda}^{mid}\leq w_{\lambda}^{mid}(T)\|(x_{i})_{i=1}^{n}\|_{\lambda}^{w}.$$
	\end{proof}
	The next proposition gives a relation between the dual of absolutely $\lambda$-summing operators and weakly mid $\lambda$-summing operators.
	\begin{pro}
		Let $X$ and $Y$ be Banach spaces, and $(\lambda,\|\cdot\|_{\lambda})$ be a reflexive $AK-BK$ space. Then $T\in W_{\lambda}^{mid}(X,Y)$ whenever $T^*\in \Pi_{\lambda}(Y^*,X^*)$.
	\end{pro}
	\begin{proof}
		Let $T:X\rightarrow Y$ such that $T^*\in\Pi_{\lambda}(Y^*,X^*)$.  Then for any $(g_{j})_{j}\in\lambda^{w^*}(Y^*)$,  $(\|T^*(g_{j})\|_{X^*})_{j}\in \lambda$.
		Also let $(x_{n})\in\lambda^{w}(X)$. Then for each $j\in\mathbb{N}$, we have
		
		\begin{align*}
			\|\left(T^*(g_{j})(x_{n})\right)_{n}\|_{\lambda}&=\|T^*(g_{j})\|\left\|\left(\frac{T^*(g_{j})}{\|T^*(g_{j})\|}(x_{n})\right)_{n}\right\|_{\lambda}\\
			&\leq\|T^*(g_{j})\|\|(x_{n})_{n}\|_{\lambda}^{w}.
		\end{align*}
		Since $\lambda$ is reflexive and hence normal with $\|\cdot\|_{\lambda}$ having norm iteration property, it follows that 
		$(T(x_{n}))_{n}\in\lambda^{mid}(Y)$.
	\end{proof}
	\section{Tensor norm defined using $\lambda^{mid}(\cdot)$}
	In this section, we define a tensor norm using the vector valued sequence spaces $\lambda^{mid}(\cdot)$ and $\lambda^{s}(\cdot)$, and establish its relation with the operator ideal $(\Pi_{\lambda}^{mid},\pi_{\lambda}^{mid})$. The results of this section mirror the approach demonstrated in \cite{molina} for absolutely $\lambda$-summing operators.

	For Banach spaces $X,Y$ and a tensor norm $\alpha$, $X\otimes_{\alpha}Y$ denotes the tensor product space endowed with the norm $\alpha$ and $X\hat{\otimes}_{\alpha}Y$ denotes its completion. We refer the reader to \cite{metric,Defant,ryan} for the terminologies associated with tensor product spaces.
	Let $u\in X\otimes Y$.  We set 
	$$\gamma_{\lambda}(u)=\inf\left\{\|(x_{i})\|_{\lambda}^{s}\|(y_{i})\|_{\lambda^{\times}}^{mid}: u=\sum_{i=1}^{n}x_{i}\otimes y_{i}\right\}.$$
	It is easily observed that $\gamma_{\lambda}$ is a quasi norm. To obtain a tensor norm, we consider the Minkowski functional of the absolute convex hull of the closed unit ball $B_{X\otimes_{\gamma_{\lambda}}Y}$ which we denote by $abs(B_{X\otimes_{\gamma_{\lambda}}Y})$.
	\begin{Lemma}
		Let $X,Y$ be Banach spaces, and 	$(\lambda,\|\cdot\|_{\lambda})$ be a normal symmetric $AK-BK$  sequence space with a $k$-symmetric norm. If $\gamma_{\lambda}^{c}$ denotes the Minkowski functional of the closed unit ball $B_{X\otimes_{\gamma_{\lambda}}Y}$,  then 
		$$\gamma_{\lambda}^{c}(u)=\inf\left\{\sum_{i=1}^{n}\|(x_{ij})_{j}\|_{\lambda}^{s}\|(y_{ij})_{j}\|_{\lambda^{\times}}^{mid}:u=\sum_{i=1}^{n}\sum_{j=1}^{m}x_{ij}\otimes y_{ij}\right\}.$$
		Moreover, we get $\gamma_{\lambda}^{c}(u)\leq\gamma_{\lambda}(u)~~~\forall u\in X\otimes Y$.
	\end{Lemma}
	\begin{proof}
		Let $u\in X\otimes Y$ and consider a representation $u=\sum_{i=1}^{n}\sum_{j=1}^{m}x_{ij}\otimes y_{ij}$.
		\begin{align*}
			u&=\sum_{i=1}^{n}\|(x_{ij})_{j}\|_{\lambda}^{s}\|(y_{ij})_{j}\|_{\lambda^{\times}}^{mid}\sum_{j=1}^{m}\frac{x_{ij}}{\|(x_{ij})_{j}\|_{\lambda}^{s}}\otimes \frac{y_{ij}}{\|(y_{ij})_{j}\|_{\lambda^{\times}}^{mid}}\\
			&\in \left(\sum_{i=1}^{n}\|(x_{ij})_{j}\|_{\lambda}^{s}\|(y_{ij})_{j}\|_{\lambda^{\times}}^{mid}\right) \text{abs}(B_{X\otimes_{\gamma_{\lambda}}Y})
		\end{align*}
		which implies $\gamma_{\lambda}^{c}(u)\leq\inf\left\{\sum\limits_{i=1}^{n}\|(x_{ij})_{j}\|_{\lambda}^{s}\|(y_{ij})_{j}\|_{\lambda^{\times}}^{mid}:u=\sum\limits_{i=1}^{n}\sum\limits_{j=1}^{m}x_{ij}\otimes y_{ij}\right\}.$

		Conversly let $\delta>0$ be such that $u\in\delta \text{abs}(B_{X\otimes_{\gamma_{\lambda}}Y})$. Then $u$ can be represented as  $u=\sum_{i=1}^{n}\alpha_{i}\sum_{j=1}^{m}x_{ij}\otimes y_{ij}$ where $\sum\limits_{i=1}^{n}|\alpha_{i}|\leq \delta$ and $\sum\limits_{j=1}^{m}x_{ij}\otimes y_{ij}\in B_{X\otimes_{\gamma_{\lambda}}Y}$.
		Then \begin{align*}
			\sum_{i=1}^{n}\|(\alpha_{i}x_{ij})_{j}\|_{\lambda}^{s}\|(y_{ij})_{j}\|_{\lambda^{\times}}^{mid}
			&\leq \sum_{i=1}^{n}|\alpha_{i}| \leq \delta
		\end{align*}
		so that, $\inf\left\{\sum\limits_{i=1}^{n}\|(x_{ij})_{j}\|_{\lambda}^{s}\|(y_{ij})_{j}\|_{\lambda^{\times}}^{mid}:u=\sum\limits_{i=1}^{n}\sum\limits_{j=1}^{m}x_{ij}\otimes y_{ij}\right\}\leq \gamma_{\lambda}^{c}(u).$ 
		
		Additionally, for any $u=\sum\limits_{i=1}^{n}\sum\limits_{j=1}^{m}x_{ij}\otimes y_{ij}$, we have \begin{align*}\sum_{i=1}^{n}\sum_{j=1}^{m}x_{ij}\otimes y_{ij}&=\|(x_{ij})_{i,j}\|_{\lambda}^{s}\|(y_{ij})_{i,j}\|_{\lambda^{\times}}^{mid}\sum_{i=1}^{n}\sum_{j=1}^{m}\frac{x_{ij}}{\|(x_{ij})_{i,j}\|_{\lambda}^{s}}\otimes \frac{y_{ij}}{\|(y_{ij})_{i,j}\|_{\lambda^{\times}}^{mid}}\\
			\implies~~\gamma_{\lambda}^{c}(u)&\leq \|(x_{ij})_{i,j}\|_{\lambda}^{s}\|(y_{ij})_{i,j}\|_{\lambda^{\times}}^{mid}
		\end{align*}
		follows from the symmetry of $\lambda$. Hence $\gamma_{\lambda}^{c}(u)\leq\gamma_{\lambda}(u)$.
		
	\end{proof}
	
	Recall from \cite{metric}, a norm $\alpha$ on $X\otimes Y$ is said to a reasonable cross norm if $\alpha$ satisfies the following properties:
	\begin{itemize}
		\item for $x\in X$ and $y\in Y$, $\alpha(x\otimes y)\leq\|x\|\|y\|$.
		\item  for $x^*\in X^*$ and $y^*\in Y^*$, $x^*\otimes y^*\in(X\otimes_{\alpha}Y)^*$ and $\|x^*\otimes y^*\|\leq\|x^*\|\|y^*\|$. 
	\end{itemize} 
	Now we prove
	\begin{pro}
		Let  $(\lambda,\|\cdot\|_{\lambda})$ be a normal symmetric $AK-BK$  sequence space with a $k$-symmetric norm such that $\|e_{1}\|_{\lambda}=1$. Then $\gamma_{\lambda}^{c}(\cdot)$ is a reasonable cross  norm.
	\end{pro}
	\begin{proof}
		For $x\otimes y\in X\otimes Y$, we have \begin{align*}\gamma_{\lambda}^{c}(x\otimes y)&=\|xe_{1}\|_{\lambda}^{s}\|ye_{1}\|_{\lambda^{\times}}^{mid} =\|x\|\sup_{T\in B_{\mathcal{L}(Y,\lambda^{\times})}}\|Ty\|_{\lambda^{\times}}\\
			&\leq \|x\|\|y\|.
		\end{align*}
		Let $x^*\in X^*$, $y\in Y^*$ and $u=\sum_{i=1}^{n}\sum_{j=1}^{m}x_{ij}\otimes y_{ij}$ be a representation of $u\in X\otimes Y$. Then
		\begin{align}\label{reasonable2}
			\left|x^*\otimes y^*(u)\right|&=\left|\sum_{i=1}^{n}\sum_{j=1}^{m}x^*(x_{ij})y^*(y_{ij})\right|\\\nonumber
			&\leq \|x^*\|\|y^*\|\sum_{i=1}^{n}\|(x_{ij})_{j}\|_{\lambda}^{w}\|(y_{ij})_{j}\|_{\lambda^{\times}}^{w}\\\nonumber
			&\leq \|x^*\|\|y^*\|\sum_{i=1}^{n}\|(x_{ij})_{j}\|_{\lambda}^{s}\|(y_{ij})_{j}\|_{\lambda^{\times}}^{mid}.
		\end{align}
		Since \eqref{reasonable2} is true  for all representations of $u$, we have $\left|x^*\otimes y^*(u)\right|\leq \|x^*\|\|y^*\|\gamma_{\lambda}^{c}(u)$. Thus, 
		$\|x^*\otimes y^*(u)\|\leq \|x^*\|\|y^*\|$.
	\end{proof}
	From the definition of $\gamma_{\lambda}^{c}$, it is clear that it is finitely generated. The uniform property of $\gamma_{\lambda}^{c}$ also follows easily. Therefore, we have
	\begin{pro}
		Let $(\lambda,\|\cdot\|_{\lambda})$ be a normal symmetric $AK-BK$  sequence space with a $k$-symmetric norm such that $\|e_{1}\|_{\lambda}=1$.	Then $\gamma_{\lambda}^{c}(\cdot)$ is a tensor  norm.
	\end{pro} 
	Now we prove that the tensor norm 	$\gamma_{\lambda}^{c}$ is associated to the maximal ideal $(\Pi_{\lambda}^{mid},\pi_{\lambda}^{mid})$ in the sense of Defant and  Floret\cite[Theorem 17.5]{Defant}. Note that for $\lambda=\ell_{p}$, we recover the relation between the tensor norm $\gamma_{p}$ and the operator ideal of absolutely mid $p$-summing operators, see\cite[Proposition 4.7]{AP}.
	\begin{thm}
		\label{gamma lambda and Pi lambda mid}
		Let $X,Y$ be Banach spaces, and $(\lambda,\|\cdot\|_{\lambda})$ be a normal symmetric $AK-BK$  sequence space such that $\|e_{1}\|_{\lambda}=1$. Then a linear operator $T:Y\rightarrow X^{*}$ is absolutely mid $\lambda^{\times}$-summing if and  only if the linear functional $\Phi_{T}$ corresponding to $T$  belongs to $(X\hat{\otimes}_{\gamma_{\lambda}^{c}} Y)^{*}$. In this case, the operator norm of $\Phi_{T}$ coincides with $\pi_{\lambda^{\times}}^{mid}(T)$. 	
	\end{thm}
	\begin{proof}
		Let $T\in\Pi_{\lambda}^{mid}(Y,X^*)$. Consider $u=\sum\limits_{i=1}^{n}\sum\limits_{j=1}^{m}x_{ij}\otimes y_{ij}\in X\otimes_{\gamma_{\lambda}^{c}}Y$. We define $\Phi_{T}:X\otimes Y\rightarrow \mathbb{K}$ as 
		$$\Phi_{T}(u)=\sum\limits_{i=1}^{n}\sum\limits_{j=1}^{m}\left<x_{ij}, Ty_{ij}\right>.$$
		Then
		\begin{align*}
			\left|\Phi_{T}(u)\right|&\leq\sum\limits_{i=1}^{n}\sum\limits_{j=1}^{m}|\left<x_{ij}, Ty_{ij}\right>|\\
			&\leq \sum\limits_{i=1}^{n}\|(Ty_{ij})_{j}\|_{\lambda^{\times}}^{s}\sum\limits_{j=1}^{m}\left|\left<x_{ij}, \frac{Ty_{ij}}{\|(Ty_{ij})_{j}\|_{\lambda^{\times}}^{s}}\right>\right|\\
			&\leq \sum\limits_{i=1}^{n}\|(Ty_{ij})_{j}\|_{\lambda^{\times}}^{s}\|(x_{ij})_{j}\|_{\lambda}^{s}\\&\leq \Pi_{\lambda^{\times}}^{mid}(T)\sum\limits_{i=1}^{n}\|(y_{ij})_{j}\|_{\lambda^{\times}}^{mid}\|(x_{ij})_{j}\|_{\lambda}^{s}
		\end{align*}
		which proves that $\Phi_{T}\in (X\otimes_{\gamma_{\lambda}^{c}}Y)^{*}$ with $\|\Phi_{T}\|\leq\Pi_{\lambda^{\times}}^{mid}(T).$
		
		Conversly, let $\Phi\in (X\otimes_{\gamma_{\lambda}^{c}}Y)^{*}$. Define $T_{\Phi}:Y\rightarrow X^*$ as $\left<T_{\Phi}(y),x\right>=\left<\Phi,x\otimes y\right>$. Then for each $\epsilon> 0$ and $(\delta_{i})_{i}\in B_{\lambda^{\times}}$, there exists $x_{i}\in B_{X}$ such that
		$\|T_{\Phi}(y_{i})\|\leq\left|\left<\Phi,x_{i}\otimes y_{i}\right>\right|+\epsilon\delta_{i}$ for each $1\leq i\leq n$. Then we have
		\begin{align*}
			\|(T_{\Phi}y_{i})_{i}\|_{\lambda^{\times}}^{s}&=\sup_{(\alpha_{i})_{i}\in B_{\lambda}}\sum\limits_{i=1}^{n}\|T_{\Phi}(y_{i})\||\alpha_{i}|\\
			&\leq \sup_{(\alpha_{i})_{i}\in B_{\lambda}}\sum\limits_{i=1}^{n}\left|\left<\Phi,x_{i}\otimes y_{i}\right>\right||\alpha_{i}| +\epsilon\\
			&\leq \|\Phi\| \|(y_{i})_{i}\|_{\lambda^{\times}}^{mid}
		\end{align*}
		follows from
		$\gamma_{\lambda}^{c}(\cdot)\leq\gamma_{\lambda}(\cdot)$.
	\end{proof}

	\textbf{ACKNOWLEDGEMENT}. The first author acknowledges  the financial support from the Department of Science and Technology, India (Grant No. DST/INSPIRE Fellowship/2019/IF190043) and second author acknowledges the financial support from the Department of Science and Technology, India(Grant No. SPG/2021/004500).
	
\end{document}